\documentclass[12pt]{article}
\usepackage[margin=1in]{geometry}
\usepackage{amsmath,amsfonts,amssymb,amsthm,enumerate,url,hyperref,pstricks}

\newtheorem{thm}{Theorem}
\newtheorem{prop}[thm]{Proposition}
\newtheorem{lem}[thm]{Lemma}

\newtheorem{conj}[thm]{Conjecture}
\theoremstyle{definition}
\newtheorem*{rem}{Remark}
\newtheorem*{defn}{Definition}
\newtheorem{ex}[thm]{Example}

\DeclareMathOperator{\Proj}{Proj}
\DeclareMathOperator*{\lcm}{lcm}
\DeclareMathOperator{\inn}{in}
\newcommand{\FF}{\mathcal{F}}
\newcommand{\PP}{\mathbb{P}}
\newcommand{\QQ}{\mathbb{Q}}
\newcommand{\RR}{\mathbb{R}}
\newcommand{\ZZ}{\mathbb{Z}}
\newcommand{\kk}{\Bbbk}
\newcommand{\floor}[1]{\lfloor #1 \rfloor}

\newcommand{\ignore}[1]{}
\def\noqed{\renewcommand{\qedsymbol}{}}

\begin{document}
\title{Canonical rings of $\QQ$-divisors on $\PP^1$}
\author{Evan O'Dorney}
\maketitle
\begin{abstract}
The canonical ring $S_D = \bigoplus_{d \geq 0} H^0(X,\floor{d D})$ of a divisor $D$ on a curve $X$ is a natural object of study; when $D$ is a $\QQ$-divisor, it has connections to projective embeddings of stacky curves and rings of modular forms. We study the generators and relations of $S_D$ for the simplest curve $X = \PP^1$. When $D$ contains at most two points, we give a complete description of $S_D$; for general $D$, we give bounds on the generators and relations. We also show that the generators (for at most five points) and a Gr\"obner basis of relations between them (for at most four points) depend only on the coefficients in the divisor $D$, not its points or the characteristic of the ground field; we conjecture that the minimal system of relations varies in a similar way. Although stated in terms of algebraic geometry, our results are proved by translating to the combinatorics of lattice points in simplices and cones.
\end{abstract}

\section{Introduction}
Let $X$ be a smooth curve over a field $\kk$, and let
\[
  D = \sum_{i=1}^n \alpha_i P_i
\]
be a $\QQ$-divisor, where the $P_i$ are distinct points of $X$. The \emph{canonical ring} of $D$, denoted $S_D$ (or $S$ when there is no ambiguity), is the graded $\kk$-algebra
\[
  \bigoplus_{d \geq 0} u^d H^0(X,\floor{d D})
\]
where it is understood that floors should be taken of the multiplicities of each point. Here $u$ is a dummy variable to keep track of the grading.

There are two motivations for studying such rings. On the purely algebro-geometric side, isomorphisms $X \to \Proj S$ give embeddings into (weighted) projective space canonically determined by $D$. When $D = K$ is a canonical divisor, \emph{Petri's theorem} \cite{SD} states that this embeds most curves $X$ as intersections of quadrics in projective space and classifies the exceptions. Other natural choices for $D$ are a canonical divisor plus a point or a small collection of points (the so-called \emph{log-canonical ring}). When $X$ is a \emph{stacky} curve---a curve with fractional points to be thought of locally as the quotient of a classical curve by a finite group acting with stabilizers---the corresponding canonical divisor contains fractional points $(1 - 1/e_i)P_i$, where the $e_i$ are the orders of the stabilizers. These canonical rings were classified exhaustively in \cite{VZB} and correspond to embeddings of $X$ into weighted projective stacks.

In addition, many of these canonical rings also arise as rings of modular forms. The graded algebra of modular forms for a Fuchsian group $\Gamma \subset \mathrm{PSL}_2(\RR)$ is isomorphic to the stacky canonical ring of the orbifold $\overline{\mathcal{H}}/\Gamma$. Moreover, if all the stabilizer orders $e_i$ are odd, then this surface also admits odd-weight modular forms, which form a ring corresponding to a $\QQ$-divisor with coefficients $(e_i - 1)/(2e_i)$. In some cases, this ring is simpler to understand than its more natural even-weight subring (see e.g. \cite{Mil}, Lemma 6.4). Fractional-weight modular forms correspond to even more exotic $\QQ$-divisors.

In this paper, we study the simplest case $X = \PP^1$ for general $\QQ$-divisors $D$. We denote by $t$ the rational coordinate on $\PP^1$; thus $S$ is a subring of $\kk(t)[u]$. We can at once exclude the cases
\begin{itemize}
  \item $\deg D < 0$, where $S_D = \kk$ is the trivial algebra; and
  \item $\deg D = 0$, where $S_D$ is the polynomial ring in a nonzero element of $H^0(\ell D)$ where $\ell$ is the minimal positive integer such that $\ell D$ is a $\ZZ$-divisor.
\end{itemize}
We will assume that $n \geq 2$; this is no restriction, since we can always add ``ghost points'' with multiplicity $0$, and the $n = 1$ case will turn out fairly easy anyway.

In the following, notation of the form $H^0(L)$ or $H^0(D)$ (for a line bundle $L$ or divisor $D$ on $\PP^1$) is understood to mean $H^0(\PP^1, L)$ or $H^0(\PP^1, O(D))$.

Our main theorems are as follows. When $n = 2$, we describe the ring $S_D$ completely (Theorems \ref{thm:1point} and \ref{thm:2point}) in terms of the rational best lower approximations to the $\alpha_i$, which are computable as convergents to continued fractions. We then show that for general $D$, the degrees of the generators and relations are bounded in terms of the lowest common denominators of the $\alpha_i$, taken $n$ and $n-1$ at a time (Theorem \ref{thm:bounds}). Finally, we ask what aspects of $S_D$ are stable under varying the $P_i$ or the characteristic of the field $\kk$ over which they are defined. We show that a stable generating set exists if
\begin{itemize}
  \item $n \leq 5$ (Theorem \ref{thm:stablegen});
  \item we look only at the generators in degrees $d \geq (2n-2)/\deg D$ (Theorem \ref{thm:xgen}(a));
  \item the $P_i$ stay away from a locus $X_n$ of bad configurations that depends only on $n$ (Theorem \ref{thm:xgen}(b)).
\end{itemize}
We show that the relations, at least as represented by a Gr\"obner basis, are stable if $n \leq 4$ (Theorem \ref{thm:stablegro}), and we conjecture that the minimal basis for the relations is stable under the same conditions that the generators are (Conjecture \ref{conj:min}).

Magma was used to assist in the discovery of theorems and conjectures; the code is included at the end of the .tex file.

\section{General remarks on the generators of $S_D$}
\label{sec:gen}
Denote by $t_i$ the section of $O(1)$ that vanishes only at the point $P_i$ (it is unique up to scaling). Then, for any integers $c_1, \ldots, c_n$ with sum $0$, the monomial $\prod_i t_i^{c_i}$ is a rational function having order of vanishing $c_i$ at $P_i$ and $0$ elsewhere. Our first proposition states that $S$ can be described in terms of such functions.

\begin{prop} \label{prop:mon}
As a $\kk$-vector space, $S$ is spanned by the monomials $u^d t_1^{c_1} t_2^{c_2} \cdots t_n^{c_n}$ satisfying
\begin{equation}
  c_i \geq -d \alpha_i \quad (1 \leq i \leq n) \quad \text{and} \quad \sum_{i} c_i = 0
  \label{eq:cone}
\end{equation}
and has as a basis those monomials which in addition satisfy
\begin{equation}
  c_i = -\floor{d \alpha_i} \quad \text{for} \quad 3 \leq i \leq n.
  \label{eq:edge}
\end{equation}
\end{prop}
\begin{proof}
The conditions \eqref{eq:cone} are clearly necessary and sufficient for a monomial $u^d \prod_i t_i^{c_i}$ to belong to $S$. Thus it is enough to prove that the monomials satisfying \eqref{eq:cone} and \eqref{eq:edge} form a basis for $S$. We can do this one graded piece $S_{d}$ at a time. Write
\[
  \floor{dD}= \sum_i b_i P_i
\]
and let $r = \sum_i b_i$.
Then
\[
  S_{d} = H^0\Big(\sum_i b_i P_i\Big) = \frac{t_1^r}{\prod_i t_i^{b_i}} H^0(r P_1).
\]
It is now sufficient to prove that $\{(t_2/t_1)^j|0 \leq j \leq r\}$ is a basis for $H^0(r P_1)$, which is true: it has $\max\{r+1,0\} = h^0(O(r))$ elements that are linearly independent because they have distinct pole orders at $P_1$.
\end{proof}

Thus, to give generators for $S$, it suffices to find a generating set for the semigroup $\Sigma$ of lattice points inside the cone $C$ of points $(d,c_1,\ldots, c_n) \in \RR^{n+1}$ satisfying \eqref{eq:cone}.

\section{One point}
\label{sec:1point}

\noindent%
\begin{minipage}{0.86\textwidth}
The most natural special case to begin our discussion is when $D = \alpha P$ is supported at only one point, which we may take to be $\infty$. Then each space $H^0(dD)$ takes the form $u^d \cdot H^0(r(\infty)) = u^d \cdot \langle 1, t, t^2, \ldots, t^r \rangle$ where $r = \floor{d \alpha}$. So as a $\kk$-vector space,
\[
  S = \langle t^c u^d \mid 0 \leq c \leq d \alpha \rangle.
\]
\begin{ex}\label{ex:1point}
Take $\alpha = 13/5$. In degree $1$, the functions $f_0 = u$, $f_1 = tu$, and $f_2 = t^2u$ appear immediately and are needed as generators. In degree $2$, the functions $t^c u^2$ for $0 \leq c \leq 4$ are products of the existing $f$'s, but $f_3 = t^5 u^2$ is new. No new generators appear in degrees $3$ and $4$, but $f_4 = t^{13} u^5$ is a generator, the first nonconstant monomial giving equality in the condition $c \leq d \alpha$. No additional generators are needed. The accompanying figure illustrates how the slopes $c/d$ approximate $\alpha$ ever more closely.
\end{ex}
\end{minipage}%
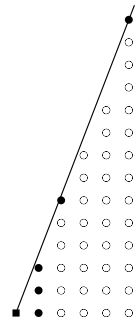
\begin{minipage}{0.14\textwidth}
\centering
%\ignore
{
\psset{unit=0.3cm,linewidth=0.05}
\begin{pspicture}(-0.5,-0.5)(5,13)
\psdots*[dotstyle=square*,dotsize=3pt](0,0)
\psdots*[dotstyle=*,dotsize=3pt](1,0)(1,1)(1,2)(2,5)(5,13)
\psdots*[dotstyle=o,dotsize=3pt]
  (2,0)(2,1)(2,2)(2,3)(2,4)
  (3,0)(3,1)(3,2)(3,3)(3,4)(3,5)(3,6)(3,7)
  (4,0)(4,1)(4,2)(4,3)(4,4)(4,5)(4,6)(4,7)(4,8)(4,9)
  (5,0)(5,1)(5,2)(5,3)(5,4)(5,5)(5,6)(5,7)(5,8)(5,9)(5,10)(5,11)(5,12)
\psline(0,0)(5.25,13.65)
\end{pspicture}
}
\end{minipage}
\begin{defn}
Let $\alpha$ be a real number. A \emph{best lower approximation} to $\alpha$ is a fraction $c/d \leq \alpha$ ($d > 0$) such that there is no fraction $c'/d'$ ($d' > 0$) with
\[
  \frac{c}{d} \leq \frac{c'}{d'} \leq \alpha \quad \text{but} \quad d' < d.
\]
A \emph{best upper approximation} to $\alpha$ is the negative of a best lower approximation to $-\alpha$.
\end{defn}
Note that, by this definition, all integers less than or equal to $\floor{\alpha}$ are best lower approximations. Also, if $\alpha$ is rational, there are but finitely many nonintegral best lower approximations, since the denominator $d$ cannot exceed that of $\alpha$. We will need one nontrivial fact.
\begin{lem}\label{lem:best}
If $c_1/d_1 < c_2/d_2$ are two consecutive best lower approximations, then $c_2d_1 - c_1d_2 = 1$, that is, the two vectors $v_i = (d_i,c_i)$ form a positively oriented $\ZZ$-basis for $\ZZ^2$.
\end{lem}
\begin{proof}
Obviously $(v_1,v_2)$ forms an oriented $\RR$-basis. If it is not a $\ZZ$-basis then there are lattice points in the angle
\[
  \angle v_1 v_2 = \{ a_1 v_1 + a_2 v_2 | a_1,a_2 \geq 0 \}
\]
besides those where $a_1$ and $a_2$ are integers. Let $w = (d,c) = a_1 v_1 + a_2 v_2$ be such a point with $d$ minimal. Then $a_1$ and $a_2$ are both less than $1$; neither of them can be zero (otherwise $c_i/d_i$ would be nonreduced) so $c_1/d_1 < c/d < c_2/d_2$. Also by comparing $w$ with $v_i + v_{i+1} - w$ we see that $d \leq (d_1 + d_2)/2 \leq d_2$. We claim that $c/d$ is a best lower approximation to $\alpha$, contradicting the consecutivity of $c_1/d_1$ and $c_2/d_2$. Suppose otherwise that $c'/d'$ is a better approximation with a smaller denominator. Then $c'/d' < c_2/d_2$ (since $c_2/d_2$ is a best lower approximation) so $c'/d'$ is a point strictly inside $\angle v_1 v_2$, contradicting the choice of $c/d$.
\end{proof}

\begin{thm}\label{thm:1point}
Let $D = \alpha (\infty)$ be a one-point $\QQ$-divisor, and let
\[
  0 = \frac{c_0}{d_0} < \frac{c_1}{d_1} < \cdots < \frac{c_r}{d_r} = \alpha
\]
be the nonnegative best lower approximations to $\alpha$. Then $S$ has a minimal presentation consisting of the $r+1$ generators $f_i = t^{c_i}u^{d_i}$ and $\binom{r}{2}$ relations of the form
\begin{equation} \label{eq:1ptrel}
  g_{ij} = f_i f_j - f_{h_{ij}}^{a_{ij}} \; (i < h_{ij} < j) \quad \text{or} \quad g_{ij} = f_i f_j - f_{h_{ij}}^{a_{ij}} f_{h_{ij}+1}^{b_{ij}} \; (i < h_{ij} < h_{ij}+1 < j)
\end{equation}
for each $(i,j)$ with $j \geq i+2$.
\end{thm}
\begin{rem}
A quadratic $f_{i+1} f_{i-1}$ whose indices differ by exactly $2$ has, according to the theorem, an expression as a monomial $f_i^{a_i}$ involving $f_i$ alone. If we assemble these exponents (which, by monotonicity of the $c$'s and $d$'s, are at least $2$) into a ``minus continued fraction''
\[
  \cfrac{1}{d_1 - \cfrac{1}{a_1 - \cfrac{1}{a_2 - \cfrac{1}{\ddots - \cfrac{1}{a_{r-1}}}}}}
\]
then it is an easy induction to prove that the convergents are the $c_i/d_i$ and hence that the entire continued fraction is the (unique) minus continued fraction expansion of $\alpha$.
\end{rem}
\begin{rem}
We can also consider the case when $\alpha$ is irrational; then $S_D$ will have infinitely many generators corresponding to the convergents of the corresponding nonterminating minus continued fraction expansion. Since, in general, the canonical ring of an $\RR$-divisor is simply an increasing union of canonical rings of $\QQ$-divisors, we will not speak much of the $\RR$-divisor case.
\end{rem}
We do not prove this theorem here, since it is a special case of the theorem for a two-point divisor to be stated in the next section (Theorem \ref{thm:2point}).

\section{Two points}
\label{sec:2points}
\begin{minipage}{0.86\textwidth}
This case is closely related to the preceding. Here $D$ is supported at two points, which we may take to be $P_1 = \infty$ and $P_2 = 0$, and hence instead of lying in $\kk[t,u]$ as before, the canonical ring resides in the Laurent polynomial ring $\kk[t,1/t,u]$ as the span of the monomials
\[
  \langle t^c u^d \mid -\alpha_2 d \leq c \leq \alpha_1 d \rangle
\]
lying in a fixed angle.

\begin{ex}
Building off of the previous example, we take $D = (13/5)(\infty) - (1/4)(0)$ and note how, in addition to the generators $tu$, $t^2u$, $t^5 u^2$, $t^{13} u^5$ approaching $13/5$ from below, we need a family $tu^2$, $tu^3$, $tu^4$ of generators approaching $1/4$ from above.
\end{ex}

\end{minipage}%
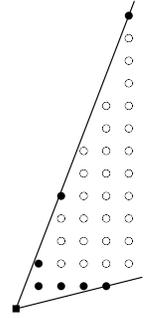
\begin{minipage}{0.14\textwidth}
\centering
%\ignore
{
\psset{unit=0.3cm,linewidth=0.05}
\begin{pspicture}(-0.5,-0.5)(5,13)
\psdots*[dotstyle=square*,dotsize=3pt](0,0)
\psdots*[dotstyle=*,dotsize=3pt](1,1)(1,2)(2,5)(5,13)(2,1)(3,1)(4,1)
\psdots*[dotstyle=o,dotsize=3pt]
  (2,2)(2,3)(2,4)
  (3,2)(3,3)(3,4)(3,5)(3,6)(3,7)
  (4,2)(4,3)(4,4)(4,5)(4,6)(4,7)(4,8)(4,9)
  (5,2)(5,3)(5,4)(5,5)(5,6)(5,7)(5,8)(5,9)(5,10)(5,11)(5,12)
\psline(0,0)(5.25,13.65)
\psline(0,0)(5.6,1.4)
\end{pspicture}
}
\end{minipage}
\begin{thm}\label{thm:2point}
Let $D = \alpha P + \beta Q$ be a two-point $\QQ$-divisor, where $\alpha + \beta \geq 0$. Let $c_0/d_0$ be a rational number between $-\beta$ and $\alpha$ (inclusive) of minimal denominator. Let
\[
  \frac{c_0}{d_0} < \frac{c_1}{d_1} < \cdots < \frac{c_r}{d_r} = \alpha
\]
be the subsequent best lower approximations of $\alpha$, and let
\[
  \frac{c_0}{d_0} > \frac{c_{-1}}{d_{-1}} > \cdots > \frac{c_{-s}}{d_{-s}} = -\beta
\]
be the best \emph{upper} approximations of $-\beta$ following $c_0/d_0$. Then $S$ has a minimal presentation consisting of the $r+s+1$ generators $f_i = t^{c_i}u^{d_i}$ (where $t$ is a rational coordinate on $\PP^1$ with a pole at $P$ and a zero at $Q$) and $\binom{r}{2}$ relations of the form
\begin{equation} \label{eq:2ptrel}
  g_{ij} = f_i f_j - f_{h_{ij}}^{a_{ij}} \; (i < h_{ij} < j) \quad \text{or} \quad g_{ij} = f_i f_j - f_{h_{ij}}^{a_{ij}} f_{h_{ij}+1}^{b_{ij}} \; (i < h_{ij} < h_{ij}+1 < j)
\end{equation}
for each $(i,j)$ with $j \geq i+2$.
\end{thm}

\begin{proof}
By construction, $c_0/d_0$ is both a best lower approximation to $\alpha$ and a best upper approximation to $\beta$.

Obviously $f_0$ is a minimal generator. If $f_i$ were not minimal for some $i > 0$, then (since $S$ is spanned by monomials) there would be two nonconstant elements $t^{c'}u^{d'}, t^{c''}u^{d''} \in S$ whose product is $f_i = t^{c_i}u^{d_i}$. Clearly both $d'$ and $d''$ are less than $d_i$. Now $c_i/d_i$ is the mediant of two fractions $c'/d'$, $c''/d''$ not exceeding $\alpha$, one of which must be at least $c_i/d_i$, contradicting the definition of best lower approximation. The proof for $i < 0$ is symmetric.

To show that no further generators are necessary, let $v_i$ be the lattice vector $(d_i,c_i) \in \ZZ^2$. Denoting by $\angle v_i v_j$ the set of all real points inside or on the sides of the indicated angle, we have a decomposition
\[
  \angle v_{-s} v_r = \angle v_{-s} v_{-s+1} \cup \angle v_{-s+1} v_{-s+2} \cup \cdots \cup v_{r-1} v_r.
\]
By Lemma \ref{lem:best}, all the lattice points in $\angle v_i v_{i+1}$ are nonnegative linear combinations of $v_i$ and $v_{i+1}$. Each monomial $t^c u^d \in S$ lies in (at least) one of these angles and hence has an expression of the form $f_h^{x} f_{h+1}^{y}$ where $x$ and $y$ are nonnegative integers, establishing the generation claim. In particular, if $j - i \geq 2$, then the product $f_i f_j$ has this form for some $h$, $i \leq h < j$ (since $v_i + v_j \in \angle v_iv_j$), giving the relations $g_{ij}$. Each of them is minimal since no other of them contains a term dividing $f_i f_j$.

As promised in the theorem, we show that $f_i$ and $f_j$ themselves do not appear in the expansion of $f_i f_j$; in other words, $v_i + v_j \in \angle v_{i+1}v_{j-1}$. Since $\{v_i, v_{i+1}\}$ is a positively oriented $\ZZ$-basis, the cross product $v_i \wedge v_{i+1}$ is $1$; also $v_{i+1} \wedge v_j = m$ is a positive integer. Then $(v_i + v_j) \wedge v_{i+1} = 1-m \leq 0$ so $v_i + v_j$ does not lie on the same side of $v_{i+1}$ as $v_i$. Symmetrically $v_i + v_j$ does not lie on the same side of $v_{j-1}$ as $v_j$.

Now, given any monomial in the $f_i$'s involving indices differing by at least $2$ can be transformed, by applying the relation $g_{ij}$ on the highest and lowest $f$'s, to a monomial in which the highest index either drops or is preserved with a lower exponent. Iterating this process eventually produces the canonical form in which indices differ by at most $1$. Thus the $g_{ij}$'s generate all relations among the $f_i$'s.
\end{proof}

\section{Bounds}
\label{sec:bounds}
For $n \geq 3$, the canonical ring of a divisor $D = \sum_{i=1}^n \alpha_i P_i$ is no longer generated by monomials in $t$ and $u$ and cannot be described as explicitly as in the foregoing cases. Moreover, the generators and relations can occur in rather high degrees, as the next example shows.
\begin{ex}\label{ex:235}
Take $D = -\frac{1}{2}P_1 + \frac{1}{3}P_2 + \frac{1}{5}P_3$, so $\deg D = 1/30 > 0$. For $1 \leq d \leq 5$, $\floor{dD}$ has negative degree and hence no elements appear. At $d=6$, $d=10$, and $d=15$, $\deg \floor{dD}$ rises to $0$, yielding three generators which we denote respectively by $x$, $y$, $z$. Up to degree $29$, there are no further elements apart from the monomials $x^2$, $xy$, etc., which also cause $\deg \floor{dD}$ to rise to $0$. Then, at degree $30$, the three monomials $x^5$, $y^3$, and $z^2$ must be squeezed into a $2$-dimensional vector space, yielding a relation $ax^5 + by^3 + cz^2 = 0$. By comparing the pole orders of the three terms at $P_1$, $P_2$, and $P_3$ one can check that $a$, $b$, and $c$ are all nonzero. Moreover, all degrees of $S$ having dimension at least $2$ have the form $H^0((30+d)D)=H^0(30D)\cdot H^0(dD)$ from which one can deduce that there are no more generators or relations.
\end{ex}

Nevertheless, we can bound the degrees of the generators and relations needed.
\begin{thm} \label{thm:bounds}
Write $\alpha_i = p_i/q_i$ in lowest terms. Let
\[
  \ell = \lcm_j q_j \quad \text{and} \quad \ell_i = \lcm_{j \neq i} q_j.
\]
Then $S$ is generated in degrees less than $\sum_i \ell_i$, with relations in degrees less than
\[
  \max \Big\{\ell + \sum_{i}{\ell_i}, 2\sum_{i}{\ell_i}\Big\}.
\]
\end{thm}
\begin{rem}
These bounds are tight at least up to a factor of $2$; to see this, consider the example
\[
  D = \frac{p_1}{2q_1} P_1 + \cdots + \frac{p_n}{2q_n} P_n
\]
where the $q_i$ are pairwise coprime odd integers and the $p_i$ are adjusted so that $\deg D = 1/\ell = 1/\big(2\prod_i q_i\big)$. Here (by an analysis like that of Example \ref{ex:235}), there are minimal generators in degrees $\ell_i$ and $\frac{1}{2}\sum_i \ell_i$ and relations in degrees $\sum_i\ell_i$ and $\ell$. However, in this example $\deg D = 1/\ell$ is very close to $0$; perhaps an even tighter bound can be found in terms of $1/\deg D$ and the $q_i$ (a generator in degree $q_i$ not vanishing at $P_i$ is clearly needed).
\end{rem}
\begin{proof}[Proof of the generator bound]
As has been previously mentioned, to find a generating set for $S$, it is enough to find a generating set for the semigroup $\Sigma$ of lattice points lying in the cone $C$ defined by the inequalities
\[
  c_1 \geq -d\alpha_1, \quad \ldots, \quad c_n \geq -d\alpha_n.
\]
Denote by $e_i$ the minimal lattice point on the ray in which all but the $i$th of these inequalities becomes equality; since $d$ must be divisible by all the $q_j$ where $j \neq i$, we find the coordinates of $e_i$ to be
\begin{align*}
  d &= \ell_i, \\
  c_j &= -\alpha_j \ell_i \quad (j \neq i), \\
  c_i &= -\sum_{j\neq i} c_j = (\deg D - \alpha_i) \ell_i
\end{align*}
Note that the $e_i$ are linearly independent ($c_i + d\alpha_i$ is a linear functional vanishing on all but one) and $C$ is the set of points expressible as $a_1 e_1 + \ldots + a_n e_n$ with coordinates $a_i \geq 0$. The $\ZZ$-points under the old coordinates map to a lattice $\Lambda$, $\ZZ^n \subseteq \Lambda \subset \QQ^n$, in the $a$-coordinates, such that $\Sigma$ maps to $\Lambda \cap \ZZ_{\geq 0}^n$. We seek a system of generators for $\Sigma$; we may take the basis vectors $e_i$ plus all nonzero $\Lambda$-points $\lambda_j = \sum_i a_{ji} e_i$ lying in the fundamental cube $0 \leq a_1,\ldots,a_n < 1$.

The degrees of these generators can be found by converting back to the original coordinates: we have $d(e_i) = \ell_i$ and hence
\[
  d(\lambda_j) = \sum_i a_{ji} \ell_i < \sum_i \ell_i
\]
as desired.
\end{proof}

\begin{rem}
We have just proved that $S$ is finitely generated. This proof gives a bound on not only the degrees of the generators but also their number. The number of $\Lambda$-points in the fundamental cube is simply its volume in the coordinates $(d,c_1,\ldots,c_{n-1})$ (the $c_n$-coordinate being redundant), which is the absolute value of the determinant
\begin{align*}
  \begin{vmatrix}
  e_1 \\ e_2 \\ \vdots \\ e_n
  \end{vmatrix}
  &= \ell_1 \ell_2 \cdots \ell_n
  \begin{vmatrix}
  1 & \deg D - \alpha_1 & -\alpha_2 & \cdots & -\alpha_{n-1} \\
  1 & -\alpha_1 & \deg D - \alpha_2 & \cdots & -\alpha_{n-1} \\
  \vdots & \vdots & \vdots & \ddots & \vdots \\
  1 & -\alpha_1 & -\alpha_2 & \cdots & \deg D - \alpha_{n-1} \\
  1 & -\alpha_1 & -\alpha_2 & \cdots & -\alpha_{n-1} \\
  \end{vmatrix} \\
  &= \ell_1 \ell_2 \cdots \ell_n
  \begin{vmatrix}
  0 & \deg D & 0 & \cdots & 0 \\
  0 & 0 & \deg D & \cdots & 0 \\
  \vdots & \vdots & \vdots & \ddots & \vdots \\
  0 & 0 & 0 & \cdots & \deg D \\
  1 & -\alpha_1 & -\alpha_1 & \cdots & -\alpha_{n-1} \\
  \end{vmatrix} \\
  &= \pm \ell_1 \ell_2 \cdots \ell_n (\deg D)^{n-1}.
\end{align*}
Removing the point $0$ and adding the generators $e_1,\ldots, e_n$ shows that $S$ is generated by at most
\[
  n-1 + \ell_1 \ell_2 \cdots \ell_n (\deg D)^{n-1}
\]
elements.
\end{rem}

\begin{proof}[Proof of the relation bound]
Returning to the theorem, we now have generators $g_1,\ldots,g_N$ for $S$ having the forms of monomials $g_i = u^{d_i} t_1^{c_{i1}} \cdots t_n^{c_{in}}$ whose corresponding lattice points generate $\Sigma$. We ask for the relations among them, that is, generators of the kernel of the induced map
\[
  \phi : \kk[x_1, \ldots, x_N] \rightarrow S.
\]
(If any of the $g_i$'s happens to be redundant, its inclusion clearly cannot decrease the degrees of the relations.) It is convenient to factor $\phi$ as a certain composition. Consider the semigroup ring
\[
  \kk[\Sigma] = \langle u^d y_1^{c_1} \ldots y_n^{c_n}
  \mid c_i \geq -d\alpha_i \text{ and } \sum_i c_i = 0 \rangle
\]
lying within the Laurent polynomial ring $\kk[u,y_1,y_1^{-1},\ldots, y_n,y_n^{-1}]$. Then $\phi$ factors as the composition of the maps
\begin{align*}
  \chi : \kk[x_1, \ldots, x_N] &\to \kk[\Sigma] \\
  x_i &\mapsto u^{d_i}y_1^{c_{i1}}y_2^{c_{i2}}\cdots y_n^{c_{in}}
\end{align*}
and
\begin{align*}
  \psi : \kk[\Sigma] &\to S \\
  u^d y_1^{c_1} \ldots y_n^{c_n} &\mapsto u^d t_1^{c_1} \ldots t_n^{c_n}.
\end{align*}
Since $\chi$ is a surjection, we have the natural exact sequence
\[
  0 \to \ker \chi \to \ker \phi \to \ker \psi \to 0.
\]
Moreover, both $\chi$ and $\psi$ are graded maps if $\kk[\Sigma]$ is graded by $u$-degree and $x_i$ is given the degree of the corresponding generator. Hence it suffices to bound the degrees of the generators of $\ker \chi$ and $\ker \psi$, which is the content of the following two lemmas.
\end{proof}
\begin{lem}
The kernel of $\chi$ is generated in degrees less than $2 \sum_i \ell_i$.
\end{lem}
\begin{proof}
Since $\chi$ takes each variable to a monomial in $\kk[\Sigma]$, its kernel is generated by equalities between monomials and the question is really one about the relations among the generators $e_j$ ($1 \leq j \leq n$) and $\lambda_j$ ($1 \leq j \leq m$) of the semigroup $\Sigma$. Letting $\lambda_0 = 0$ for convenience, we note that every $\lambda_i + \lambda_j$ ($1 \leq i,j \leq m$) has the form $\lambda_k$ ($0 \leq k \leq m$) plus a sum of some distinct $e_h$'s. These are a family of relations of degree less than $2 \sum_i \deg e_i = 2 \sum_i \ell_i$; and they generate all relations by allowing one to put any element of $\Sigma$ into the canonical form $\lambda_j + \sum_i a_i e_i$.
\end{proof}
\begin{lem}
The kernel of $\psi$ is generated in degrees less than $\ell + \sum_i \ell_i$.
\end{lem}
\begin{proof}
We first claim that $\ker \psi$ is spanned as a $\kk$-vector space by elements of the form
\begin{equation}
  (y_j - \beta_j y_1 - \gamma_j y_2) \cdot u^d y_1^{b_1} \cdots y_n^{b_n}
  \label{eq:rel}
\end{equation}
for $3 \leq j \leq n$, $b_i \geq -\alpha_i d$, and $\sum_i b_i = -1$, where $\beta_j$ and $\gamma_j$ are the coordinates of $t_j$ with respect to the basis $(t_1, t_2)$ of the two-dimensional space $H^0(O(1))$. If we have a monomial $f = u^d t_1^{c_1} \cdots t_n^{c_n} \in S$ where any of the $c_j$'s ($3 \leq j \leq n$) strictly exceeds $-\floor{d\alpha_j}$, then we can put $b_j = c_j-1$ and all other $b_i = c_i$ to write $f$ as a linear combination of monomials with strictly smaller values of $\sum_{i \geq 3} c_j$. Iterating this process, we arrive at a linear combination of the $r+1$ monomials with all exponents minimal except those of $t_1$ and $t_2$, which form a basis for $S_d$ by Proposition \ref{prop:mon}.

(We note that $\psi$ is an isomorphism for $n = 2$, one of the circumstances making this case exceptionally tractable.)

It remains to bound the generation degree of the $S$-module with generators \eqref{eq:rel}. They are parametrized by integer points $(d,b_1,\ldots,b_n)$ in the cone $C'$ defined by the inequalities
\[
  b_i \geq -\alpha_i d \quad \text{with} \quad \sum_i b_i = -1.
\]
Over $\RR$, the cone $C'$ is simply a translation of $C$ by the unique vector
\[
  \epsilon = \left(\frac{1}{\deg D}, \frac{-\alpha_1}{\deg D}, \ldots, \frac{-\alpha_n}{\deg D}\right)
\]
where equality holds in each of the defining inequalities of $C'$; thus all its lattice points have the form $\epsilon + \sum_{i} a_i e_i$ where $a_i \geq 0$ and $\{e_1,\ldots,e_n\}$ is the aforementioned $\RR_{\geq0}$-basis of $C$. If any $a_i$ is $1$ or greater, the corresponding relation is redundant, as it can be generated from a simpler one by multiplication by $\psi(e_i) \in S$. So any minimal relation has degree less than
\[
  \deg (\epsilon + e_1 + \cdots + e_n) = \frac{1}{\deg D} + \sum_i \ell_i \leq \ell + \sum_i \ell_i
\]
as desired.
\end{proof}

\section{Stability of generation}
\label{sec:stabgen}
An issue that obviously does not come up when $n \leq 3$ is how the ring $S$ depends on the points $P_i$ if these are allowed to move while their multiplicities $\alpha_i$ remain fixed. In general, the detection of a linear relation among three or more monomials $u^d t_1^{c_1} \cdots t_n^{c_n}$ can be a thorny task.
\begin{ex}\label{ex:char}
  Let $D = 2P_1 + 0P_2 + 0P_3$ (the zeros are to artificially inflate the supply of monomials; small rational numbers could be used instead). Then $S_1 = f_0H^0(O(2))$ where $f_0 = ut_1^{-2}$ is fixed. Suppose we try the basis
\[
  t_1^2, \quad t_2^2, \quad t_3^2
\]
for $O(2)$. Normalizing so that $P_1 = \infty$, $P_2 = 0$, and $P_3 = 1$, we recognize these (up to a factor of $t_1^2$) as the functions $1$, $t^2$, and $(t-1)^2$, which are linearly independent---if and only if the ground field $\kk$ does not have characteristic two!
\end{ex}
\begin{ex}\label{ex:harm}
  If $D = 2P_1 + 0P_2 + 0P_3 + 0P_4$ has one more ghost point, we still have $S_1 = f_0H^0(O(2))$ and can try
\[
  t_1^2, \quad t_2^2, \quad t_3t_4
\]
as a basis for $O(2)$. Letting $P_1 = \infty$, $P_2 = 0$, $P_3 = 1$, and $P_4 = \lambda$, we see that a linear dependency among the corresponding functions $1$, $t^2$, $(t-1)(t-\lambda)$ exists if and only if $\lambda = -1$, that is, if the points $P_1$ and $P_2$ harmonically separate $P_3$ and $P_4$.
\end{ex}

In these examples, the issue can be eliminated by choosing the basis $\{t_1^2, t_1t_2, t_2^2\}$. However, the following more complicated example shows that the number and degrees of the generators (which are invariants of the ring $S$) can also depend on the points $P_i$.
\begin{ex}\label{ex:chords}
Let $D = -\frac{1}{2}(P_1 + P_2) + \frac{1}{3}(P_3 + P_4) + \frac{1}{5}(P_5 + P_6)$. As in Example \ref{ex:235}, we find three minimal generators $x \in S_{6}$, $y \in S_{10}$, $z \in S_{15}$ that generate all elements of degree $29$ or less. However, in degree $30$, the elements $f_1 = x^5$, $f_2 = y^3$, and $f_3 = z^2$ now live in a vector space of dimension $\deg(30D) + 1 = 3$. The line bundle $O(30D)$ can be identified with $O(2)$, which is the pullback of $O_{\PP^2}(1)$ under the Veronese map $v : \PP^1 \to \PP^2$; the $f_i$ are sections which lift uniquely to $O_{\PP^2}(1)$. On counting up the poles and zeros, it becomes apparent that the two zeros of $f_i$ are at $P_{2i-1}$ and $P_{2i}$. Hence the image $\tilde{f}_i \in H^0(O_{\PP^2}(1))$ is a linear form vanishing on the line $\overline{v(P_{2i-1})v(P_{2i})}$. We see that $S$ has a relation (and hence a generator) in degree $30$ if and only if these three linear forms are linearly dependent, which in turn occurs exactly when the three chords $v(P_{2i-1})v(P_{2i})$ are concurrent.

With a bit of computation, we can calculate the full minimal presentation of the canonical ring in each of these two cases; the results are summarized here:
\begin{itemize}
  \item Generic case: three generators in degrees $2$, $3$, $5$; one relation in degree $60$.
  \item Special case: four generators in degrees $2$, $3$, $5$, $30$; two relations in degrees $30$ and $60$.
\end{itemize}
Thus the number of generators and relations, as well as their degrees, is affected by the point locations.
\end{ex}
We show that, at least regarding the degrees of the generators, the six points used in this example are minimal.
\begin{defn}
A number (or $S$-element, etc{.}) depending on a $\QQ$-divisor $D = \sum_{i} \alpha_i P_i$ is \emph{stable} if its value (or form) is independent of the points $P_i$ (as long as no two coincide) and the characteristic of the ground field $\kk$.
\end{defn} 
\begin{thm}\label{thm:stablegen}
If $\alpha_1,\ldots,\alpha_n$ are rational numbers and $n \leq 5$, then the canonical ring of the divisor $D = \sum_{\alpha_i P_i}$ has a stable generating set of monomials $u^d t_1^{c_1}\cdots t_n^{c_n}$.
\end{thm}
\begin{proof}
Fix $d \geq 1$. A minimal set of generators in degree $d$ is a basis for the quotient of $S_d$ by the space of pregenerated functions
\[
  \sum_{c=1}^{\floor{d/2}} S_{c} S_{d-c}.
\]
If $S_{c}$ or $S_{d-c}$ is zero (and this can be checked stably) then so is their product; otherwise
\begin{align*}
  S_{c} S_{d-c} &= u^d H^0(\floor{cD}) H^0(\floor{(d-c)D}) \\
  &=  u^d H^0(\floor{cD}+ \floor{(d-c)D}) \\
  &= u^d H^0\Big(\sum_{i=1}^n (\floor{c\alpha_i} + \floor{(d-c)\alpha_i})P_i\Big).
\end{align*}
(Here we are using the easy fact that, on $\PP^1$, the natural map $H^0(O(a)) \otimes H^0(O(b)) \to H^0(O(a+b))$ is surjective if $a,b \geq 0$.)
Note that
\[
  d\alpha_i - 2 < \floor{c\alpha_i} + \floor{(d-c)\alpha_i} \leq d\alpha_i
\]
or, since the middle expression is an integer,
\[
  \floor{d\alpha_i} - 1 \leq \floor{c\alpha_i} + \floor{(d-c)\alpha_i} \leq \floor{d\alpha_i}.
\]
Hence $\floor{cD}+ \floor{(d-c)D}) = \floor{dD} - \sum_{i\in A_c} P_i$ for some subset $A_c$ of $\{1,\ldots,n\}$. Define
\[
  V_A = u^d H^0(\floor{dD} - \sum_{i\in A} P_i) \subseteq S_d.
\]
We now seek a basis for a space of the form
\[
  S_d / \sum_{S \in \FF} V_S
\]
for some family $\FF \subseteq 2^{\{1,\ldots,n\}}$. Since the line bundle $L = O(\floor{dD})$ is isomorphic to $O(r)$ ($r = \deg \floor{dD}$) and the monomials we desire in our basis are exactly (up to scaling) the sections of $L$ whose zeros are at the points $P_i$, we have reduced the theorem to the following lemma.
\end{proof}

\begin{lem}\label{lem:cores}
Let $n$ $(2 \leq n \leq 5)$ and $r$ be integers, and let $P_1,\ldots,P_n$ be distinct points of $\PP^1$. Denote by $V_i$ the subspace of sections in $V = H^0(O(r))$ that vanish at $P_i$; let
\[
  V_A = \bigcap_{i \in A} V_i = H^0\Big(O(r)\Big(-\sum_{i\in A} P_i\Big)\Big)
\]
for $A \subseteq \{1,\ldots,n\}$. Then for any family $\FF \subseteq 2^{\{1,\ldots,n\}}$, the quotient of $V$ by the subspace
\[
  V_\FF = \bigcup_{A \in \FF} V_A
\]
has a stable basis of monomials $t_1^{c_1}\cdots t_n^{c_n}$.
\end{lem}
\begin{proof}
Assume not; and assume that $|\FF|$ is as small as possible. In particular, each $V_A$ is nonempty, so
\[
  |A| \leq r
\]
for each $A \in \FF$. On the other hand, the sets $A$ cannot be so small that $V_A + V_B$ could be consolidated into $V_{A\cap B}$. This happens if the rightmost map in the exact sequence
\begin{equation}
  0 \to V_{A \cup B} \xrightarrow{\text{diag}} V_{A} \oplus V_{B} \xrightarrow{-} V_{A \cap B} \label{eq:2sets}
\end{equation}
is surjective, which always happens if $|A \cup B| \leq r+1$ (for then each $V_C$ in the diagram has dimension $r - |C| + 1$ and the alternating sum of the dimensions vanishes). Thus
\begin{equation}\label{eq:union}
  |A \cup B| \geq r+2
\end{equation}
for all distinct $A, B \in \FF$. Now suppose $\FF$ contains three elements $A,B,C$. We have
\begin{align*}
  n &\geq |A \cup B \cup C| \\
  &= |A \cup B| + |A \cup C| + |B \cup C| - |A| - |B| - |C| + |A \cap B \cap C| \\
  &\geq 3(r+2) - 3r + 0 \\
  &= 6,
\end{align*}
a contradiction.

So $\FF$ has at most $2$ elements and we can finish the proof in each case.
\begin{itemize}
  \item If $\FF = \emptyset$, then $V/V_{\FF} = V$ has a basis $t_2^r, t_2^{r-1}t_1,\ldots, t_1^r$.
  \item If $\FF = \{A\}$, we may assume $1 \in A$. $V_{\FF}$ has dimension $r - |A| + 1$ and contains sections with zeros of order $1,\ldots,r - |A| + 1$ at $P_1$. Thus any set of monomials of degree $r$ containing $t_1$ to the exponents $0, r - |A| + 2, r - |A| + 3, \ldots, r$ is a basis for $V/V_{\FF}$.
  \item If $\FF = \{A,B\}$, note that either $A$ or $B$ has size exactly $r$ since otherwise
  \begin{align*}
    2(r-1) \geq |A| + |B| &\geq |A \cup B| \geq r+2 \\
    r &\geq 4 \\
    n &\geq |A \cup B| \geq r+2 \geq 6.
  \end{align*}
  Without loss of generality $|B| = r$. Also $A$ cannot be contained in $B$ so we may assume $1 \in A \setminus B$. Now $V_B$ is one-dimensional and its one basis element does not vanish at $P_1$, while $V_A$ has a basis of $r - |A| + 1$ elements vanishing to orders $1,\ldots,r - |A| + 1$ at $P_1$. Hence a system of monomials vanishing to orders $r - |A| + 2, r - |A| + 3, \ldots, r$ furnishes the desired basis. \qedhere
\end{itemize}
\end{proof}
Although we have stated the preceding theorem for the case $n \leq 5$, our methods are more widely applicable.
\begin{thm}\label{thm:xgen}
\begin{enumerate}[$(a)$]
  \item[]
  \item The generators of $S_D$ in degrees greater than or equal to $(2n-2)/\deg D$ can be given by stable monomials.
  \item Fix the ground field $\kk$. For each $n \geq 1$, there is a proper closed subvariety $X_n \subset (\PP^1)^n$ with the following property: if $\alpha_1,\ldots,\alpha_n$ are fixed rational numbers (or even real numbers), then the number of generators of $S_D$ in each degree does not depend on the points $P_i$ as long as $(P_1,\ldots,P_n) \notin X_n$.
\end{enumerate}
\end{thm}
\begin{proof}
The key to both parts is that if $r \geq n-1$, the proof of Lemma \ref{lem:cores} still goes through, as follows. First replace $\FF$ by the minimal family yielding the same space $V_\FF$. Then $|\FF| \leq 1$ since otherwise we have \eqref{eq:union} which is a contradiction (since $|A \cup B| \leq n \leq r+1$). Then the proof ends up in either the $|\FF| = 0$ or the $|\FF| = 1$ case, neither of which uses the hypothesis that $n \leq 5$.
\begin{proof}[Proof of $(a)$] If $d \geq (2n-2)/\deg D$, then
\[
  r = \deg \floor{d D} > d \deg D - n \geq (2n - 2) - n = n - 2,
\]
so $r \geq n-1$.
\noqed
\end{proof}
\begin{proof}[Proof of $(b)$] Here is a crude argument. Whether a set of monomials is a minimal generating set in degree $d$ is tantamount to whether a collection $m_1,\ldots,m_r$ of degree-$r$ monomials in $t_1,\ldots,t_r$ is a basis for $V/V_\FF$ considered as a quotient of $O(r)$. As we have just seen, such a basis can be chosen stably if $r \geq n-1$. For each $r = 0,1,\ldots,n-2$ and each $\FF \subseteq 2^{\{1,\ldots,n\}}$, we can choose a generic basis for $V/V_\FF$ simply by removing elements from the basis $t_1^r, t_1^{r-1}t_2,\ldots, t_2^r$ of $V$. Let $X_{r,\FF}$ be the variety of point locations where the chosen set ceases to be a basis. Then, away from the finite union $X = \bigcup_{r,\FF} X_{r,\FF}$, the ring $S_D$ has a stable monomial basis.

The same variety $X_n$ works for $\alpha_i$ real, since we can replace them by rational approximations (say by rounding down to the nearest multiple of $1/d!$) without changing the structure of the canonical ring up to degree $d$.
\end{proof}
\noqed
\end{proof}
\begin{rem}
As $\kk$ varies, the variety $X_n$ is of course always defined by the same integer polynomial equations, reduced mod $\operatorname{char} \kk$, and thus might be thought of more naturally as a scheme over $\ZZ$.
\end{rem}
\section{Stability of relations}
\label{sec:stabrel}
We now arrive at the problem of determining how the relations of $S$ depend on characteristic and point location. Following the approach of Schreyer \cite{Sch}, we consider the relations as calculated by first finding a Gr\"obner basis for the relation ideal $I = \ker \phi$ and then whittling it down to minimality. For the reader's benefit, we review the definition of a Gr\"obner basis while setting up our conventions. Let $g_1, g_2,\ldots, g_N$ be a minimal generating set, ordered in increasing order by degree (we will soon specify the ordering among generators of the same degree). Then equip the ring $R = \kk[x_1,x_2,\ldots,x_m]$ with the \emph{reverse lexicographical ordering} $\prec$ on monomials given by sorting them from least to greatest like words in a dictionary:
\[
  x_1 \prec x_1^4 \prec x_1^2 x_3 \prec x_1 x_2^2 \prec x_2\text{, \quad etc.}
\]
We then define the \emph{initial term} $\inn_\prec f$ of any nonzero element $f \in R$ as its ``largest'' or ``last'' term under this ordering. By definition, a \emph{Gr\"obner basis} for $I$ is a generating set whose initial terms form a minimal basis for $\inn_\prec I$ (the ideal that is the $\kk$-span of initial terms of elements of $I$).
\begin{rem}
Since $I$ is homogeneous with respect to the grading on the $x$'s inherited from the $g$'s, the Gr\"obner basis will be no different if we use the slightly more familiar ``grevlex,'' or graded reverse lexicographical, order in which higher-degree monomials always lead lower ones and the aforementioned ``revlex'' ordering is only applied to monomials of the same degree.
\end{rem}
\begin{thm} \label{thm:stablegro}
  If $n \leq 4$, then the relation ideal of the stable generators $g_1,\ldots, g_N$ found in Theorem \ref{thm:stablegen} (ordered, in each degree, by reverse order of mention in the proof of Lemma \ref{lem:cores}) has a Gr\"obner basis with stable leading terms.
\end{thm}
\begin{proof}
First note that it is sufficient to prove that $\inn_\prec I$ is stable, for then the leading terms of the Gr\"obner basis are simply the monomials in $\inn_\prec I$ not divisible by any other monomial in $\inn_\prec I$. Let
\[
  m_1 \prec m_2 \prec \cdots \prec m_M
\]
be the monomials of degree $d$ in the $x_i$ and consider the nested sequence of subspaces
\[
  0 = W_0 \subseteq W_1 \subseteq \cdots \subseteq W_M = S_d
\]
where $W_k$ is the $\kk$-span of $m_1,\ldots,m_k$. Note that $\dim W_k - \dim W_{k-1}$ is always either $0$ or $1$ and is $0$ if and only if $m_k \in \inn_\prec I$. Thus it is sufficient to prove that the dimension of each $W_k$ is stable. We begin by describing $W_k$ as a subspace of $S_d = u^d H^0(\floor{dD})$.

\begin{lem}\label{lem:h0sum}
There are $(\ZZ$-$)$divisors $D_1,\ldots,D_\ell$ $(\ell \geq 0)$, supported on $\{P_1,\ldots,P_n\}$, such that
\begin{equation}\label{eq:h0sum}
  u^{-d}W_k = H^0(D_1) + \cdots + H^0(D_\ell)
\end{equation}
and such that
\begin{equation}\label{eq:aura}
  D_v \leq D_u + \sum_i P_i \text{\quad whenever\quad} u < v.
\end{equation}
\end{lem}

\noindent
\begin{minipage}{0.76\textwidth}
\begin{rem}
Pictorially, the monomials in $H^0(O(r))$ form an $(n-1)$-simplex with side $r$ whose vertices are the pure powers $t_i^r$; the subspaces $H^0(D_i) \subseteq u^{-d}S_d$ can be viewed as subsimplices with the same orientation. The condition \eqref{eq:aura} says that the simplex corresponding to $D_v$ lies within the neighborhood formed by expanding each face of the simplex of $D_u$ by one lattice point. An example for $n=3$, $r=7$ is shown. (The overlapping of the simplices is a sign that the expression $\sum_i H^0(D_i)$ is not minimal---compare \eqref{eq:2divs} below.)
\end{rem}
\end{minipage}%
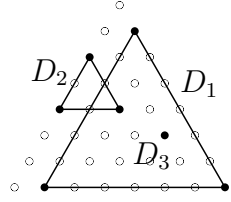
\begin{minipage}{0.24\textwidth}
\centering
%\ignore
{
\psset{unit=0.4cm,linewidth=0.05}
\begin{pspicture}(-0.5,-0.5)(7.5,6.5)
\newcommand{\sqr}{0.866}
\psdots*[dotstyle=*,dotsize=3pt](1,0)(7,0)(4,5.196)(1.5,2.598)(3.5,2.598)(2.5,4.330)(5,1.732)
\psdots*[dotstyle=o,dotsize=3pt]
  (0,0)(2,0)(3,0)(4,0)(5,0)(6,0)
  (0.5,0.866)(1.5,0.866)(2.5,0.866)(3.5,0.866)(4.5,0.866)(5.5,0.866)(6.5,0.866)
  (1,1.732)(2,1.732)(3,1.732)(4,1.732)(6,1.732)
  (2.5,2.598)(4.5,2.598)(5.5,2.598)
  (2,3.464)(3,3.464)(4,3.464)(5,3.464)
  (3.5,4.330)(4.5,4.330)
  (3,5.196)
  (3.5,6.062)
\pspolygon(1,0)(7,0)(4,5.196)
\pspolygon(1.5,2.598)(3.5,2.598)(2.5,4.330)
\rput[bl](5.5,3){$D_1$}
\rput[br](1.8,3.4){$D_2$}
\rput[tr](5.2,1.5){$D_3$}
\end{pspicture}
}
\end{minipage}
\begin{proof}
We will induct on $d$, the base case being $W_0$ written as the empty sum (for all $d$, including $d=0$). Let $m_k = g_{k_1} g_{k_2} \cdots g_{k_s}$ where $k_1 \leq k_2 \leq \cdots \leq k_s$, and let $g_{k_1}$ have degree $e$. The monomials that precede or equal $m_k$ can be divided into three (possibly overlapping) categories:
\begin{enumerate}
  \item All monomials containing a $g_i$ of degree less than $e$;
  \item All monomials containing a $g_i$ of degree $e$ that precedes $g_{k_1}$;
  \item All monomials of the form $g_{k_1} m'$ where $m'$ is a monomial of degree $d - e$ that precedes or equals $g_{k_2} \cdots g_{k_s}$.
\end{enumerate}
Correspondingly we get a three-term decomposition
\begin{equation}\label{eq:wk}
  W_k = \sum_{c < e} S_c S_{d-c} + \sum_{\substack{i < k_1 \\ \deg g_i = e}} g_i S_{d-e} + g_{k_1} W'
\end{equation}
where $W'$ is the space $W_{k'}$ in degree $d - e$ corresponding to the monomial $g_{k_2} \cdots g_{k_s}$. By induction, the last term has an expression $H^0(D'_1) + \cdots + H^0(D'_{\ell'})$ of the desired form (the multiplication by $g_{k_1}$ simply translates each divisor by a degree-0 constant). Also, we recall from the proof of Theorem \ref{thm:stablegen} that each term of the left sum has the form $u^d H^0(E)$ with $E = \floor{dD} - \sum_{i\in A}P_i$. Note that $E + \sum_i P_i \geq \floor{dD}$, so the terms $H^0(E)$ can be placed at the beginning of the desired expression without possibly violating \eqref{eq:aura}.

To finish the proof it is necessary to study the middle term of \eqref{eq:wk}, which in turn depends on the ordering of the generators $g_u, g_{u+1}, \ldots, g_v$ within degree $e$. By examining the proof of Lemma \ref{lem:cores} carefully, we observe that there is a point $P_j$ (in the lemma $j=1$, but in general $j$ changes with $e$) such that
\begin{itemize}
  \item $g_u,\ldots, g_v$ have strictly decreasing order of vanishing at $P_j$;
  \item if $g_i$ has order of vanishing $m$ at $P_j$ (as a section of $\floor{eD}$), then all elements of $S_e$ that vanish to order greater than $m$ (i.e{.} all elements of $u^e H^0(\floor{eD} - (m+1)P_j)$) lie in the subring generated by $g_u,g_{u+1},\ldots,g_{i-1}$ and $S_1,S_2,\ldots,S_{e-1}$.
\end{itemize}
Taking $i = k_1$ in the second condition, we can write
\begin{align*}
  W_k &= \sum_{c < e} S_c S_{d-c} + \sum_{u \leq i < k_1} g_i S_{d-e} + g_{k_1} W' \\
  &= \sum_{c < e} S_c S_{d-c} + u^e H^0(\floor{eD} - (m+1)P_j) S_{d-e} + g_{k_1} W' \\
  &= \sum_{c < e} S_c S_{d-c} + u^d H^0(\floor{eD} + \floor{(d-e)D} - (m+1)P_j) + g_{k_1} W'.
\end{align*}
We finally note that $g_{k_1} W' \subseteq g_{k_1} S_{d-e} \subseteq u^d H^0(\floor{eD} + \floor{(d-e)D} - m P_j)$, so \eqref{eq:aura} will be satisfied if the terms arising inductively from $W'$ come last.
\end{proof}
To prove the theorem, it suffices to show that for $n \leq 4$ (and there is no harm in assuming $n=4$) any space of the form \eqref{eq:h0sum} satisfying \eqref{eq:aura} has stable dimension. We begin by simplifying the expression so that $\ell$ is minimal. This means that no term vanishes (so each $\deg D_i \geq 0$) and also that it is not possible to combine two terms as follows. If $D_u = \sum_i a_i P_i$ and $D_v = \sum_i b_i P_i$ are two divisors, define
\[
  D_u \cap D_v = \sum_i \min\{a_i,b_i\} P_i \quad \text{and} \quad
  D_u \cup D_v = \sum_i \max\{a_i,b_i\} P_i.
\]
Then we have the exact sequence
\begin{equation}\label{eq:2divs}
  0 \to H^0(D_u \cap D_v) \to H^0(D_u) \oplus H^0(D_v) \to H^0(D_u \cup D_v).
\end{equation}
If the rightmost map is surjective, we can replace the two terms $D_u$, $D_v$ with $D_u \cup D_v$ (and place it at the $u$th spot, assuming $u < v$) without violating \eqref{eq:aura}. By a simple dimension count, this always happens if $\deg (D_u \cap D_v) \geq -1$, so we have
\[
  \deg (D_u \cap D_v) \leq -2.
\]
We can derive another piece of information from \eqref{eq:2divs}: $\dim W_k$ is stable when $\ell = 2$. Therefore we can assume $\ell \geq 3$. Let $u < v < w$ be three distinct indices and consider the corresponding divisors $D_u = \sum_i a_i P_i$, $D_v = \sum_i b_i P_i$, $D_w = \sum_i c_i P_i$. We have the conditions
\begin{minipage}[t]{0.22\textwidth}
\begin{gather}
  b_i \leq a_i + 1 \label{eq:ba} \\
  c_i \leq a_i + 1 \label{eq:ca} \\
  c_i \leq b_i + 1 \label{eq:cb}
\end{gather}
\end{minipage}\hfill
\begin{minipage}[t]{0.25\textwidth}
\begin{gather}
  \sum_i a_i \geq 0 \label{eq:sa} \\
  \sum_i b_i \geq 0 \label{eq:sb} \\
  \sum_i c_i \geq 0 \label{eq:sc}
\end{gather}
\end{minipage}\hfill
\begin{minipage}[t]{0.35\textwidth}
\begin{gather}
  \sum_i \min\{a_i,b_i\} \leq -2 \label{eq:mab} \\
  \sum_i \min\{a_i,c_i\} \leq -2 \label{eq:mac} \\
  \sum_i \min\{b_i,c_i\} \leq -2 \label{eq:mbc}.
\end{gather}
\end{minipage}
\begin{lem} \label{lem:slog}
The only integer solutions to \eqref{eq:ba}--\eqref{eq:mbc} when $n=4$ are, up to permuting the $i$'s and adding the same constants $d_i$ $(\sum_i d_i = 0)$ to $a_i$, $b_i$, and $c_i$, given by the following table.
\begin{equation}
\begin{tabular}{c|cccc}
$i$ & $1$ & $2$ & $3$ & $4$ \\ \hline
$a_i$ & $\geq 2$ & $0$ & $-1$ & $-1$ \\
$b_i$ & $0$ & $0$ & $0$ & $0$ \\
$c_i$ & $1$ & $1$ & $0$ & $-2$
\end{tabular}
\label{tab}
\end{equation}
\end{lem}
\begin{proof}
We begin by adding together the inequalities \eqref{eq:sb} and \eqref{eq:mab} to get
\begin{align*}
  2 &\leq \sum_i \big(b_i - \min\{a_i,b_i\}\big) \\
  &= \sum_i \max\{b_i - a_i, 0\}.
\end{align*}
Similarly, we add together \eqref{eq:sc} and \eqref{eq:mbc} to get
\[
  2 \leq \sum_i \max\{c_i - b_i, 0\}.
\]
Each term of these two sums is either $0$ and $1$, and both cannot be $1$ for the same $i$ by \eqref{eq:ca}. So equality holds, and we may permute indices such that
\[
  c_1 - b_1 = c_2 - b_2 = b_3 - a_3 = b_4 - a_4 = 1.
\]
Also \eqref{eq:sb} is an equality so we may normalize the $b_i$ to be all $0$, and now all but $a_1, a_2, c_3, c_4$ have the values indicated in \eqref{tab}. We have $a_1 + a_2 \geq 2$ (by \eqref{eq:sa}) and may assume $a_1 \geq 1$; we have $c_3 + c_4 = -2$ (now that \eqref{eq:sc} is an equality) and may assume $c_3 \geq -1$. The only condition still unused is \eqref{eq:mac}:
\begin{align*}
-2 &\geq \sum_i {\min\{a_i,c_i\}} \\
   &\geq 1 + a_2 + (-1) + c_4 \\
   &= a_2 - 2 - c_3 \\ 
   &\geq (b_2 - 1) - 2 - (b_3 - 1) = -2.
\end{align*}
Equality holds everywhere, and \eqref{tab} quickly follows.
\end{proof}
Returning to the theorem, we assume first that $\ell \geq 4$ and focus on the first four divisors $D_1, D_2, D_3, D_4$. The difference $D_2 - D_4$ can be calculated in two ways: applying Lemma \ref{lem:slog} to $D_1,D_2,D_4$ shows that
\[
  D_2 - D_4 = \sum_i (b_i - c_i)P_i = -P_{\sigma(1)} - P_{\sigma(2)} + 2P_{\sigma(4)}
\]
for some permutation $\sigma$; using $D_2,D_3,D_4$ instead gives
\[
  D_2 - D_4 = \sum_i (a_i - c_i)P_i = (a_1 - 1) P_{\tau(1)} - P_{\tau(2)} - P_{\tau(3)} + P_{\tau(4)}
\]
for some permutation $\tau$ and integer $a_1 \geq 2$. Because the former value has a zero coefficient and the latter does not, we have a contradiction.

So $\ell = 3$, and it suffices to show that $H^0(D_1) + H^0(D_2) + H^0(D_3)$ has stable dimension when $D_1$, $D_2$, $D_3$ exactly equal the divisors in the table \eqref{tab}. Pick a basis for each $H^0(D_i)$ consisting of functions with distinct orders at $P_1$. These orders will be $0$ for $D_2$, $1$ for $D_3$, and $2,3,\ldots,a_1$ for $D_1$: all distinct, implying that
\[
  \dim\Big(\sum_i H^0(D_i)\Big) = \sum_i \dim H^0(D_i) = a_1 + 1
\]
is stable.
\end{proof}
\section{Open questions}
\label{sec:open}
This work leaves many directions open for future research.
\begin{enumerate}
\item The outstanding question is how to compute, or prove stability of, the \emph{minimal} relations among the generators. In contrast to the Gr\"obner basis, the number and degrees of these depend on neither a monomial ordering nor a generating set and thus are invariants of the divisor $D$. For $n\leq 2$ the relations \eqref{eq:2ptrel} are obviously minimal simply because of their distinctive quadratic leading terms, but even for a three-point example as simple as $D = -\frac{1}{3}P_1 + \frac{1}{2}P_2 + \frac{1}{2}P_3$ the Gr\"obner basis need not be minimal. The dimension-counting method used to prove Theorem \ref{thm:stablegro} is not well adapted to determining what terms the relations contain, let alone whether they can be used to eliminate one another. The following conjectures are plausible on the basis of numerical data:
\begin{conj}\label{conj:min}
\begin{enumerate}[$(a)$]
  \item[]
  \item For $n \leq 5$, the degrees of the minimal relations are stable.
  \item For all $n$, the degrees of the minimal relations are independent of the points $P_i$ when $(P_1,\ldots,P_n)$ lies outside the locus $X_n$ of exceptional generation in Theorem \ref{thm:xgen}$(b)$.
\end{enumerate}
\end{conj}
\item The variety $X_n$ in Theorem \ref{thm:xgen}(b) remains mysterious. When $n \leq 5$, Theorem \ref{thm:xgen}(b) tells us that $X_n$ is nothing more than the ``fat diagonal'' where two of the $P_i$ coincide. We conjecture that $X_6$ consists only of the fat diagonal and the locus on which Example \ref{ex:chords} exhibits exceptional behavior. As $n \to \infty$, how many components does $X_n$ have? What are their degrees?
\item Throughout this paper, we have been using special coordinates $u^d t_1^{c_1} \cdots t_n^{c_n}$ for our generators. Following \cite{VZB}, it is equally possible to use generic elements in the relevant degrees. What does the generic initial ideal of relations look like? Is its corresponding Gr\"obner basis more likely to be minimal than for special coordinates?
\item If the $\alpha_i$ are real numbers, the ring $S_D$ will no longer be finitely generated and we can ask about the density of generators and relations as $d \to \infty$. For $n \leq 2$ we are led into Khinchin's theory of the asymptotics of random continued fractions; for $n \geq 3$ the form of the answer is far from clear.
\item A natural extension is to curves $X \neq \PP^1$. Owing to our constant use of the fact that any line bundle on $\PP^1$ is isomorphic to some $O(r)$, we expect the Picard group and Jacobian of $X$ to take center stage. Two stark differences from the $\PP^1$ case are worthy of note:
\begin{itemize}
  \item Already for $D = P$, the ring $S_D$ ceases to be stable: its graded dimensions encode whether $P$ is a Weierstrass point and, if so, what its Weierstrass semigroup is.
  \item It is generally impossible to bound the degrees of the generators in terms of the $\alpha_i$ alone (as in Theorem \ref{thm:bounds}). For instance, if $X$ is an elliptic curve and $D = P - Q$ is a degree-$0$ divisor, then $S_d$ has dimension $1$ or $0$ according to whether $d(P-Q)$ is zero in the elliptic group law: thus $S$ can be made to have a generator in degree $d$ by making $P-Q$ a torsion point of exact order $d$.
\end{itemize}
Likewise, extensions to higher-dimensional varieties such as $\PP^m$ are natural.
\item A step beyond generators and relations is the computation of \emph{syzygies:} considering the relation ideal $I$ as a module over the free generator algebra $F = \kk[x_1,\ldots,x_n]$, to take the minimal free resolution
\[
  \cdots \to F_2 \to F_1 \to I.
\]
Much progress has been made in recent years on \emph{Green's conjecture}, which relates these syzygy modules to the geometry of the curve in the case that $D = K$ is a canonical divisor. Here one could hope for a bound similar to Theorem \ref{thm:bounds} on the length and term dimensions (``Betti numbers'') in the resolution.
\end{enumerate}

\section{Acknowledgments}
This research was performed at the Emory University Research Experience for Undergraduates (REU), an NSF-sponsored program, under the mentorship of Ken Ono and David Zureick-Brown. I would also like to thank David Yang for helpful conversations.

\bibliography{canqdiv.bib}
\bibliographystyle{alpha}
\end{document}